\documentclass[11pt,reqno,a4paper]{amsart}
\usepackage[toc,page]{appendix}

\usepackage[utf8]{inputenc}
\usepackage{esint}
\usepackage{MnSymbol}
\usepackage{enumitem}
\usepackage[english]{babel}
\usepackage{amsmath}
\usepackage{mathtools}
\usepackage{thm-restate}
\usepackage{comment}
\usepackage{svg}

\newtheorem{theorem}{Theorem}[section]
\newtheorem{assumption}[theorem]{Assumption}
\newtheorem*{theorem*}{Theorem}
\newtheorem{proposition}[theorem]{Proposition}
\newtheorem{lemma}[theorem]{Lemma}
\newtheorem{corollary}[theorem]{Corollary}

\theoremstyle{definition}

\theoremstyle{remark}
\newtheorem{remark}[theorem]{Remark}
\newtheorem{example}[theorem]{Example}
\numberwithin{equation}{section}

\usepackage[
colorlinks,
pdfpagelabels,
pdfstartview = FitH,
bookmarksopen = true,
bookmarksnumbered = true,
linkcolor = blue,
plainpages = false,
hypertexnames = false,
citecolor = red] {hyperref}

\hypersetup{
linktoc=page
}

\usepackage[capitalize]{cleveref}
\crefname{enumi}{}{}
\crefname{equation}{}{}

\crefname{assumption}{assumption}{assumptions}
\crefname{assumption}{Assumption}{Assumptions}

\makeatletter
\def\@tocline#1#2#3#4#5#6#7{\relax
\ifnum #1>\c@tocdepth 
\else
\par \addpenalty\@secpenalty\addvspace{#2}%
\begingroup \hyphenpenalty\@M
\@ifempty{#4}{%
\@tempdima\csname r@tocindent\number#1\endcsname\relax
}{%
\@tempdima#4\relax
}%
\parindent\z@ \leftskip#3\relax \advance\leftskip\@tempdima\relax
\rightskip\@pnumwidth plus4em \parfillskip-\@pnumwidth
#5\leavevmode\hskip-\@tempdima
\ifcase #1
\or\or \hskip 1em \or \hskip 2em \else \hskip 3em \fi%
#6\nobreak\relax
\dotfill\hbox to\@pnumwidth{\@tocpagenum{#7}}\par
\nobreak
\endgroup
\fi}
\makeatother

\def \E {\mathbb{E}}

\renewcommand{\tilde}{\widetilde}

\newcommand{\eps}{\varepsilon}

\newcommand{\abs}[1]{\ensuremath{\vert #1 \vert}}  
\newcommand{\norm}[1]{\ensuremath{\Vert #1 \Vert}} 

\allowdisplaybreaks

\begin{document}

\title[Continuity of processes and fields]{Necessary and sufficient conditions for continuity of hypercontractive processes and fields}
\author[Nummi]{Patrik Nummi}
\address{Patrik Nummi,
Department of Mathematics and Statistics,
University of Helsinki,
Finland}
\author[Viitasaari]{Lauri Viitasaari}
\address{Lauri Viitasaari,
Department of Mathematics,
Uppsala University,
Sweden}

\date{\today}

\thanks{PN acknowledges support by the Academy of Finland and the European Research Council (ERC) under the European Union's Horizon 2020 research and innovation programme (grant agreement No 818437).}
\maketitle

\begin{abstract}
Sample path properties of random processes \textcolor{black}{are} an interesting and extensively studied topic especially in the case of Gaussian processes. In this article\textcolor{black}{,} we study \textcolor{black}{the} continuity properties of hypercontractive fields, providing natural extensions for some known Gaussian results beyond Gaussianity. Our results \textcolor{black}{apply} to both random processes and random fields alike.
\end{abstract}



\allowdisplaybreaks

\section{Introduction}
Sample path regularity of random processes and fields is an important and extensively studied topic in the literature, and can be used as a tool, for example, to study convergence of certain statistical estimators or to study convergence of numerical schemes for stochastic partial differential equations. Indeed, various application areas arise from the fact that sharp modulus of continuity estimates can be translated to supremum tail bounds, providing sharp tail behaviour of the supremum.

The problem is particularly well-studied in the case of Gaussian processes and fields. One of the earliest result\textcolor{black}{s} in this direction is a sufficient condition due to Fernique \cite{Fernique} who provided a sufficient condition for the sample path continuity involving the increment metric 
$d_X(s,t) = \left[\E(X_s-X_t)^2\right]^{1/2}$. Later on, Dudley \cite{Dudley1,Dudley2} \textcolor{black}{established} a necessary condition for the continuity of a Gaussian process $X$ by using a metric entropy. While in general Dudley's condition is not sufficient, in the case of stationary Gaussian processes it turns out to be necessary and sufficient. Finally, general necessary and sufficient conditions for general Gaussian processes $X$ were obtained by Talagrand \cite{Talagrand}, in terms of metric entropies. Finally, while Talagrand's necessary and sufficient condition is rather complicated, a simple necessary and sufficient condition for H\"older continuity of Gaussian processes in terms of incremental increment metric 
$d_X(s,t) = \left[\E(X_s-X_t)^2\right]^{1/2}$ was obtained in \cite{lauri2014}, where the authors proved that the celebrated general Kolmogorov\textcolor{black}{--}Chentshov criterion for H\"older continuity is both necessary and sufficient condition for Gaussian processes.

While the topic is widely studied in the Gaussian case, the literature on modulus of continuity beyond Gaussianity is more limited. One general approach to obtain modulus of continuity for processes is to use Garsia\textcolor{black}{--}Rodemich\textcolor{black}{--}Rumsey lemma \cite{Garsia-Rodemich-Rumsey-1970}. A multiparameter version was provided in \cite{hu-le2013} where the authors obtained modulus of continuity estimates and joint H\"older continuity of solutions to certain stochastic partial differential equations driven by Gaussian noise. Finally, we mention a closely related article \cite{Viens-Vizcarra2007}, where the authors studied the case of so-called sub-$n$th Gaussian processes or Wiener processes, where $n$ is an arbitrary integer. Such processes arise, roughly speaking, from processes of form $H_n(X_t)$, where $X$ is a Gaussian process and $H_n$ is the $n$th order Hermite polynomial. More precisely, in \cite{Viens-Vizcarra2007} the authors provided sufficient conditions for continuity and, in the spirit of Talagrand, estimates for the tail of the supremum via metric entropy by using Malliavin calculus which is well-suited to the setting of a Gaussian Wiener space. 

In this article we study sample path continuity for general hypercontractive processes and fields. That is, we assume that \textcolor{black}{the} higher\textcolor{black}{-}order moments of the increments satisfy
$$
\E|X_t-X_s|^p \leq C(p)\left[\E|X_t-X_s|^2\right]^{p/2},
$$
\textcolor{black}{which allows} us to deduce conditions in terms of the simple incremental increment metric $\E|X_t-X_s|^2$. We show how the growth of the constant $C(p)$ translates into certain exponential moment bounds and tail estimates for the supremum. As a corollary, we extend the necessary and sufficient Kolmogorov\textcolor{black}{--}Chentshov criterion for the H\"older continuity beyond the Gaussian case studied in \cite{lauri2014}. It is worth \textcolor{black}{emphasizing} that, while our results can be used to cover and extend the known results on the Gaussian case \cite{lauri2014} and in the case of sub-$n$th processes \cite{Viens-Vizcarra2007}, our results require only hypercontractivity and can be used in a more general framework where, e.g. Malliavin calculus is not at our disposal. \textcolor{black}{In our setting, our moment growth condition, Assumption \ref{assu:hyper-basic}, corresponds to the sub-Weibull distributions (see \cite{vladimirova2020}) which generalise the widely studied and applied sub-Gaussian and sub-exponential distributions}. Finally, our results extend naturally to random fields, and can be e.g. used to study joint H\"older continuity, in the spirit of  \cite{hu-le2013}.

The rest of the article is organised as follows. In Section \ref{sec:process} we introduce our notation and main results, while all the proofs and auxiliary lemmas are postponed to Section \ref{sec:proofs}. 

\section{Necessary and sufficient conditions for continuity of hypercontractive processes and fields}
\label{sec:process}
We consider stochastic processes and fields, respectively, given by $X = (X_t)_{t\in K}$, where $K = [0,1]$ or $K=[0,1]^n$. We make use of the following hypercontractivity assumption:
\begin{assumption}
\label{assu:hyper-basic}
We suppose that for all $p\geq 1$ we have
\begin{equation}
\label{eq:hyper-process}
\E|X_t-X_s|^p \leq C_0^p p^{p\iota}\left[\E|X_t-X_s|^2\right]^{\frac{p}{2}},
\end{equation}
where $C_0>0$ is a generic fixed constant and $\iota\geq 0$ is a given parameter.
\end{assumption}
\begin{remark}
Condition \eqref{eq:hyper-process} implies that the distribution of the increment $X_t-X_s$ follows a sub-Weibull distribution with parameter $\iota$, see \cite{vladimirova2020}. This class is a generalisation of sub-Gaussian distributions ($\iota = 1/2$) and sub-exponential distributions ($\iota = 1$), and as such also contains random variables with bounded support. We note that a sub-exponentially distributed random variable is essentially a sub-Gaussian random variable squared. Both of these classes are studied extensively in the literature and have numerous applications; for details, we recommend the excellent textbook \cite{vershynin2018}, and references therein. The sub-Weibull distribution has recently found widespread application within data science, for instance in dynamic factor models in econometrics \cite{barigozzi2020}, deep learning and stochastic optimization \cite{wood2023}, and in Bayesian neural networks \cite{vladimirova2020}.
\end{remark}
\begin{remark}
In general \textcolor{black}{we could state hypercontractivity as} 
$$
\E|X_t-X_s|^p \leq C(p)\left[\E|X_t-X_s|^2\right]^{\frac{p}{2}},
$$
where $C(p)$ is a constant depending solely on $p$. One could state our results by using a more general form of $C(p)$ with suitable growth in $p$, but this would result in an additional layer of notational complexity. \textcolor{black}{For simplicity and to make the connection to sub-Weibull distributions more explicit, we restrict ourselves to the case $C(p) \leq C^p p^{p\iota}$. }
\end{remark}
\begin{remark}
It is worth to note that our assumption implies the exponential moment assumption (up to an unimportant constant) of \cite{Viens-Vizcarra2007}, in the case $\iota = n$ is an integer, cf. proof of Theorem \ref{thm:sufficient-process}.
\end{remark}
\textcolor{black}{The following examples cover a large spectrum of situations that arise in mathematical modelling and should convince the reader that the class of processes satisfying Assumption \ref{assu:hyper-basic} is large. The first three examples cover various processes arising naturally in Gaussian analysis, while Example \ref{ex:non-Gaussian} reveals that hypercontractivity can be achieved in many situations, whether or not one can apply Gaussian analysis and Malliavin calculus.}
\begin{example}
\label{ex:Gaussian}
If $X$ is Gaussian, then it is well-known that 
$$
\E|X_t-X_s|^p =  \frac{1}{\sqrt{\pi}}2^{\frac{p}{2}}\Gamma\left(\frac{p+1}{2}\right)\left[\E|X_t-X_s|^2\right]^{\frac{p}{2}},
$$
where $\Gamma$ is the Gamma function. By Stirling's approximation, we have
$$
\Gamma(x+1) \sim \sqrt{2\pi x}e^{-x}x^x
$$
for large $x$, and hence we may choose $\iota = \frac12$ in Assumption \ref{assu:hyper-basic}. 
\end{example}
\begin{example}
\label{ex:wiener-chaos}
Let \(\mathcal{H}\) be a separable, real Hilbert space, and \(Z = \{Z(h): h \in \mathcal{H}\} \) an isonormal Gaussian process on \(\mathcal{H}\).
We define the \(nth\) Hermite polynomial as \(H_0(x)=1\) and for \(n\geq 1\) by 
\[
H_n(x) = (-1)^n e^{x^2/2} \frac{d}{dx}e^{-x^2/2}.
\]
Denote by \(\mathcal{H}_p\) the linear space generated by the class \(\{ H_p(Z(h)): p \geq 0, \, h \in \mathcal{H}, \Vert h\Vert_\mathcal{h} = 1\}.\) This linear space is called the \(pth\) Wiener chaos of \(Z\). Then it is known (see, e.g. \cite{nourdin-peccati2012}) that 
\[
\left[\mathbb{E}\lvert F \rvert^q \right]^{1/q} \leq \left[\mathbb{E}\lvert F \rvert^r\right]^{1/r} \leq \left(\frac{r-1}{q-1}\right)^{p/2} \left[\mathbb{E}\lvert F \rvert^q \right]^{1/q}.
\]
Hence if $X$ is a process living in a fixed Wiener chaos of order $p$ (or in a finite linear combination of them with $p$ as the highest chaos), we have \eqref{eq:hyper-process} with $\iota = \frac{p}{2}$ and recover the processes studied in \cite{Viens-Vizcarra2007}. In particular, we recover the Gaussian case by setting $p=1$. 
\end{example}
\begin{example}
\label{ex:SDE}
Solutions to several kinds of stochastic differential equations driven by Gaussian processes are expected to satisfy Assumption \ref{assu:hyper-basic}, while at the same these objects live in infinite amount of Wiener chaoses (that is, have non-finite chaos decompositions). For example, under certain technical conditions, solutions to certain stochastic differential equations driven by the fractional Brownian motion satisfies Assumption \ref{assu:hyper-basic} with $\iota =\frac12$, see \cite[Condition (ii) of Theorem 5.15, Proof of Theorem 5.16, and the references therein]{Baudoin}.
\end{example}
\begin{example}
\label{ex:non-Gaussian}
Next, consider $Z_t = X_tY_t$, where $X$ and $Y$ both satisfy Assumption \ref{assu:hyper-basic} with parameters $C_{0,X}$ (resp. $C_{0,Y}$) and $\iota_X$ (resp. $\iota_Y$), and for simplicity assume that both $X$ and $Y$ start at zero. Then it follows from a straightforward application of the Cauchy-Schwartz inequality that $Z$ also satisfies Assumption \ref{assu:hyper-basic} with $C_{0,Z}= 2^{\iota_X+\iota_Y}C_{0,X}C_{0,Y}$, and $\iota_Z=\iota_X + \iota_Y$. Moreover, a linear combination of processes satisfying Assumption \ref{assu:hyper-basic} also satisfies the said assumption, see \cite[Proposition 3.1] {vladimirova2020}. This covers, for instance, stochastic examples related to fractional splines, see \cite{unser2000}; see also \cite{monje2010,sheng2011} for applications in signal processing.
\end{example}
\textcolor{black}{We remark that these examples extend naturally to the case of random fields $X = (X_t)_{t\in [0,1]^n}$ by considering Gaussian fields (cf. Example \ref{ex:Gaussian}), fields living in a fixed Wiener chaos (cf. Example \ref{ex:wiener-chaos}), solutions to certain stochastic partial differential equation models as in \cite{hu-le2013} (cf. Example \ref{ex:SDE}), or spline smoothing in multiparameter setting (cf. Example \ref{ex:non-Gaussian}). We also provide the following multiparameter example arising from the theory of partial differential equations.}

\textcolor{black}{We recall that if $X$ is a stochastic process (resp. random field) on a probability space $(\Omega, \mathcal{F}, P)$  with an index set $T$, for any fixed random parameter $\omega \in \Omega$, the mapping $t \mapsto X_t(\omega) =: X_t$ defines the sample path. If $t \to t_0$ implies that $X_t\to X_{t_0}$ almost surely, we say that $X$ is continuous at $t_0$. As usual, when this property holds everywhere on the index set $T$, we simply say that $X$ is a continuous stochastic process (resp. random field).}

We begin with the following general result providing the modulus of continuity and certain moment estimates. 
\begin{theorem}
\label{thm:sufficient-process}
Suppose that a continuous $X = (X_t)_{t\in [0,1]}$ satisfies Assumption \ref{assu:hyper-basic} and suppose that $\E( X_t-X_s)^2\leq \rho(|t-s|)^2$ for some non-decreasing, non-negative continuous function $\rho$, with $\rho(0)=0$.   Then 
\begin{align}
\label{eq:continuity-general-process}
|X_t - X_s | \leq 8 \int_0^{|s - t|} \beta^{-\iota}\left(\log\left(\frac{4B}{u^2}\right)\right)^{\iota}d\rho(u)
\end{align}
for any $\beta>0$, where
$$
B = \int_0^1\int_0^1 \exp\left(\beta\left(\frac{|X_t-X_s|}{\rho(|t-s|)}\right)^{\frac{1}{\iota}}\right)dsdt.
$$
Moreover, for $\beta \in \left(0,\frac{e\iota}{C_0^{\frac{1}{\iota}}}\right)$, the random variable $B$ has finite $p$-moments for all $p$ satisfying 
$$
1 \leq p < \frac{e\iota}{\beta C_0^{\frac{1}{\iota}}}.
$$
\end{theorem}
\begin{remark}
As in \cite{Viens-Vizcarra2007}, one could first study the existence of a continuous version in terms of metric entropies. That is, by considering the number $N_\epsilon$ of balls required to cover the interval $[0,T]$ \textcolor{black}{with respect to the} metric $\left[\E(X_t-X_s)^2\right]^{1/2}$. Then one obtains continuity provided that
$$
\int_0^\infty \left|\log N_\epsilon\right|^{\iota} d\epsilon < \infty.
$$
As we are interested in the modulus of continuity, we assume continuity a priori. 
\end{remark}
We obtain immediately the following corollary.
\begin{corollary}
\label{cor:holder-process}
Suppose that $X = (X_t)_{t\in [0,1]}$ satisfies Assumption \ref{assu:hyper-basic} and suppose that $\E (X_t-X_s)^2 \leq |t-s|^{2\alpha}$\textcolor{black}{, for $\alpha \in (0,1]$}. Let $\beta \in \left(0,\frac{e\iota}{C_0^{\frac{1}{\iota}}}\right)$. Then there exists a random variable $C(\omega)=C(\beta,\omega)$ satisfying $\E \exp(\beta_0 C(\omega)^{1/\iota})<\infty$ for any $\beta_0$ satisfying 
\begin{equation}
    \label{eq:beta_0-process}
\beta_0 < \frac{e\iota}{\left(8C_0\cdot 3^{\max(\iota-1,0)}\right)^{\frac{1}{\iota}}}
\end{equation}
and a deterministic constant $C_{d}=C_d(\beta)$ such that
\begin{align}
\label{eq:holder-process}
|X_t - X_s | \leq C(\omega)|t-s|^{\alpha} + C_{d}|t-s|^{\alpha}\left(\log \frac{1}{|t-s|}\right)^{\iota}.
\end{align}
In particular, we have
\begin{equation}
\label{eq:lil}
\limsup_{|t-s|\to 0} \frac{|X_t-X_s|}{|t-s|^\alpha\left(\log \frac{1}{|t-s|}\right)^{\iota}} \leq C_{d}.
\end{equation}
\end{corollary}
\begin{corollary} \label{cor:sufficient-process}
    Suppose $X = (X_t)_{t \in [0,1]}$ satisfies Assumption \ref{assu:hyper-basic} and assume that $\E (X_t-X_s)^2 \leq |t-s|^{2\alpha}$\textcolor{black}{, for $\alpha \in (0,1]$}. Let $\beta \in \left(0,\frac{e\iota}{C_0^{\frac{1}{\iota}}}\right)$. Then for any interval $I\subset [0,1]$ and any $s\in I$ we have
\begin{align*}
    P\left(\sup_{t \in I} \abs{X_t - X_s}  \geq u\abs{I}^\alpha + C_de^{-\alpha \iota}\iota^\iota\right) \leq C(\beta_0) e^{-\beta_0 u^{1/\iota} },
\end{align*}
where $$C(\beta_0) = C(\beta_0, \beta) = \E\textcolor{black}{\left(4B^{\beta_0 \beta^{-1}\left(8 \cdot 3^{\max(\iota-1,0\textcolor{black}{)}}\right)^{1/\iota}  }\right)}  < \infty$$ for any $\beta_0$ satisfying \eqref{eq:beta_0-process}. 
\end{corollary}
\begin{remark}
The above result is close in spirit to \cite[Theorem 3.1]{Viens-Vizcarra2007}, and generalises naturally the well-known exponential decay of the supremum of Gaussian processes, in which case we would obtain  
$$
P(\sup_{t\in I}|X_t-X_s|\geq u|I|^\alpha +C_1) \leq C_2e^{-c_3u^2}
$$
for constants $C_1,C_2$, and $C_3$.
\end{remark}
As our final main theorem, we obtain the following characterisation of H\"older continuity: under Assumption \ref{assu:hyper-basic}, Kolmogorov continuity criterion is a necessary and sufficient  condition for H\"older continuity. This extends the results of \cite{lauri2014} beyond Gaussian processes and covers, in particular, processes living in a finite sum of Wiener chaoses, cf. Example \ref{ex:wiener-chaos}.
\begin{theorem}
\label{thm:holder-process-iff}
Suppose that $X = (X_t)_{t\in [0,1]}$ satisfies Assumption \ref{assu:hyper-basic}. Then $X$ is H\"older continuous of any order $\gamma<\alpha$, i.e. for any $\epsilon>0$ 
$$
|X_t-X_s| \leq C_{\epsilon}(\omega)|t-s|^{\alpha-\epsilon},
$$
if and only if for any $\epsilon>0$ we have
\begin{equation}
\label{eq:holder-iff}
\E(X_t-X_s)^2 \leq C_\epsilon|t-s|^{2\alpha-\epsilon}.
\end{equation}
Moreover, in this case the H\"older constant $C_\epsilon(\omega)$ of $X$ satisfies 
\begin{equation}
\label{eq:exp-moments}
\E \exp\left(\beta C_\epsilon(\omega)^{\frac{1}{\iota}}\right) < \infty
\end{equation}
for small enough $\beta>0$ which depends only on $C_0$, $\alpha$, $\iota$, and $\epsilon$.
\end{theorem}
Our results extend naturally to the case of random fields $X = (X_t)_{t\in [0,1]^n}$. We begin with the following result, which extends Theorem \ref{thm:holder-process-iff} to the case of fields in a natural way.
\begin{proposition}
\label{prop:field-holder}
Suppose that $X=(X_t)_{t\in [0,1]^n}$ satisfies Assumption \ref{assu:hyper-basic}. Then $X$ is H\"older continuous of any order $\gamma<\alpha$, i.e. for any $\epsilon > 0$,  
$$
|X_t-X_s| \leq C_\epsilon(\omega)\textcolor{black}{\norm{t-s}_{}}^{\alpha -\epsilon},
$$ 
if and only if for any $\epsilon>0$ we have
\begin{align}\label{eq:holder-iff-multi-ass2}
\E(X_t-X_s)^2 \leq C_\epsilon \textcolor{black}{\norm{t-s}}^{2\alpha-\epsilon}.
\end{align}

Moreover, in this case the H\"older constant of $X$ satisfies 
\begin{equation*}
\E \exp\left(\beta C_\epsilon(\omega)^{\frac{1}{\iota}}\right) < \infty
\end{equation*}
for small enough $\beta>0$ which depends only on $C_0$, $\alpha$, $\iota$, and $\epsilon$.
\end{proposition}

If one considers rectangular increments and joint continuity, we first need some notation, taken from \cite{hu-le2013}.

Let \( x = (x_1, ...,x_n)\) and \(y = (y_1,...,y_n)\) be two elements in \(\mathbb{R}^d\). For each integer \(k = 1,2,..,n,\) we define 
\[
V_{k,y}x := (x_1,...,x_{k-1},y_k,x_{k+1},...,x_n).
\]
Let \(f\) be a \(\mathbb{R}^m\)-valued map on \(\mathbb{R}^n\). We define the operator \(V_{k,y}\) acting on \(f\) as follows:
\[
V_{k,y}f(x) := f(V_{k,y}\textcolor{black}{x}).
\]
It is simple to verify that \(V_{k,y}V_{k,y}f(x) = V_{k,y}f(x)\) and that \(V_{k,y}V_{l,y}f(x) = V_{l,y}V_{k,y}f(x)\) for any \(f\).

Next, we define the joint (rectangular) increment of a function \(f\) on an $n$-dimensional rectangle, 
\begin{align} \label{joint-rect-increment-defn}
\Box_y^n f(x) = \prod_{k=1}^n (I-V_{k,y})f(x),
\end{align}
where \(I\) denotes the identity operator.

For a random field $X=(X_t)_{t\in [0,1]^n}$, let \(d_X(t,s) := \sqrt{\mathbb{E}\lvert \Box^n_t X(s) \rvert^2 }.\) We assume that the following condition, analogous to Assumption \ref{assu:hyper-basic}, is satisfied.
\begin{assumption}
\label{assu:hyper-field}
We suppose that for all $p\geq 1$ we have
\begin{equation}
\label{eq:hyper-field}
\E|\Box_{\textcolor{black}{t}}^n X(s)|^p \leq C_0^p p^{p\iota}d^p_X(t,s),
\end{equation}
where $C_0>0$ is a generic fixed constant and $\iota\geq 0$ is a given parameter.
\end{assumption}
\begin{theorem}
\label{thm:sufficient-field}
Suppose that $X=(X_t)_{t\in [0,1]^n}$ is continuous and satisfies Assumption \ref{assu:hyper-field} and suppose that $d_X(t,s) \leq \prod_{j=1}^n\rho_j(|t_j-s_j|)$. Then 
\begin{align*}
|\Box^n_t X(s)| \leq 8^{n} \int_0^{|s_1 - t_1|}\ldots \int_0^{|s_n - t_n|}\beta^{-\iota}\left(\log\left(\frac{4^nB}{u_1^2\ldots u_n^2}\right)\right)^{\iota}d\rho_1(u_1)\ldots d\rho_n(u_n)
\end{align*}
for any $\beta>0$,
$$
B = \int_{[0,1]^n} \int_{[0,1]^n} \exp\left(\beta\left(\frac{|\Box^n_t X(s)|}{\prod_{j=1}^n\rho_j(|t_j-s_j|)}\right)^{\frac{1}{\iota}}\right)dsdt.
$$
Moreover, for $\beta \in \left(0,\frac{e\iota}{C_0^{\frac{1}{\iota}}}\right)$, the random variable $B$ has finite $p$-moments for all $p$ satisfying 
$$
1 \leq p < \frac{e\iota}{\beta C_0^{\frac{1}{\iota}}}.
$$
\end{theorem}
The following corollary is analogous to Corollary \ref{cor:holder-process}, and extends some of the results in \cite{hu-le2013} beyond Gaussianity. 
\begin{corollary}
\label{cor:holder-field}
Suppose that $X=(X_t)_{t\in [0,1]^n}$ satisfies \ref{assu:hyper-field} and suppose that $d^2_X(t,s) \leq \prod_{j=1}^n|t_j-s_j|^{2\alpha_j}$. Let $\beta \in \left(0,\frac{e\iota}{C_0^{\frac{1}{\iota}}}\right)$. Then there exists a random variable $C(\omega)=C(\beta,\omega)$ satisfying $\E \exp(\beta_0 C^{1/\iota}(\omega))<\infty$ for any $\beta_0$ satisfying 
\begin{equation}
    \label{eq:beta_0-field}
\beta_0 < \frac{e\iota}{\left(8^nC_0\cdot 3^{\max(\iota-1,0)}\right)^{\frac{1}{\iota}}}
\end{equation}
and a deterministic constant $C_{d}=C_d(\beta)$, such that
\begin{align*}
|\Box^n_t X(s)| \leq C(\omega)\prod_{j=1}^n|t_j-s_j|^{\alpha_j} + C_{d}\prod_{j=1}^n|t_j-s_j|^{\alpha_j}\left(\log \frac{1}{\prod_{j=1}^n|t_j-s_j|}\right)^{\iota}.
\end{align*}
In particular, we have
\begin{equation*}
\limsup_{\max_j|t_j-s_j|\to 0} \frac{|\Box^n_t X(s)|}{\prod_{j=1}^n|t_j-s_j|^{\alpha_j}\left(\log \frac{1}{\prod_{j=1}^n|t_j-s_j|}\right)^{\iota}} \leq C_{d}.
\end{equation*}
\end{corollary}
\begin{corollary} \label{cor:sufficient-field}
    Suppose that $X=(X_t)_{t\in [0,1]^n}$ satisfies \ref{assu:hyper-field} and suppose that $d^2_X(t,s) \leq \prod_{j=1}^n|t_j-s_j|^{2\alpha_j}$. Let $\beta \in \left(0,\frac{e\iota}{C_0^{\frac{1}{\iota}}}\right)$. Then for any intervals $I_j \subset [0,1]$ and any $s \in I = I_1 \times \ldots I_n$ we have
\begin{align*}
    P\left(\sup_{t \in I} |\Box^n_t X(s)| \geq u\prod_{j=1}^n \abs{I_j}^{\alpha_j} + \tilde{C}\right) \leq C(\beta_0) e^{-\beta_0 u^{1/\iota} },
\end{align*}
where $$C(\beta_0) = C(\beta_0, \beta) = \E\textcolor{black}{\left[4B^{\beta_0 \beta^{-1}\left(8 \cdot 3^{\max(\iota-1,0})\right)^{1/\iota}  } \right]} < \infty$$ for any $\beta_0$ satisfying \eqref{eq:beta_0-field} and 
\begin{align}\label{eq:c-tilde}
\tilde{C} = C_d\max_{0\leq x_j\leq 1}\prod_{j=1}^n |x_j|^{\alpha_j}\left(\log \frac{1}{\prod_{j=1}^n |x_j|}\right)^{\iota}.    
\end{align}
\end{corollary}
Similarly, Theorem \ref{thm:holder-process-iff} extends in a natural manner to the multiparameter case. Again, the proof is analogous to the proof of Theorem \ref{thm:holder-process-iff} and is thus left to the reader.
\begin{theorem}
\label{thm:holder-field-iff}
Suppose that $X=(X_t)_{t\in [0,1]^n}$ satisfies Assumption \ref{assu:hyper-field}. Then $X$ is jointly H\"older continuous of any order $\gamma = (\gamma_1,\ldots,\gamma_n)$ with $\gamma_j<\alpha_j$, i.e. for any $\epsilon =(\epsilon_1, ..., \epsilon_n)$ with $\epsilon_i > 0$,
$$
|\Box^n_t X(s)| \leq C(\omega)\prod_{j=1}^n |t_j-s_j|^{\gamma_j-\epsilon_j},
$$
 if and only if for any $\epsilon =(\epsilon_1, ..., \epsilon_n)$ with $\epsilon_i > 0$ we have
$$
d^2_X(t,s) \leq C_\epsilon\prod_{j=1}^n|t_j-s_j|^{2\alpha_j-\epsilon_j}.
$$
Moreover, in this case the H\"older constant of $X$ satisfies 
\begin{equation*}
\E \exp\left(\beta C_\epsilon(\omega)^{\frac{1}{\iota}}\right) < \infty
\end{equation*}
for small enough $\beta>0$ which depends only on $C_0$, $\alpha$, $\iota$, and $\epsilon$.
\end{theorem}

\section{Proofs}
\label{sec:proofs}
Our results in Section \ref{sec:process} are based on the following Garsia-Rodemich-Rumsey inequality \cite{Garsia-Rodemich-Rumsey-1970} and its multiparameter extension \cite{hu-le2013}.
\begin{proposition}
\label{prop:GRR}
Let $\Psi(u)$ be a non-negative, even function on $(- \infty, \infty)$ and $p(u)$ be a non-negative, even function on $[-1,1]$. Assume both $p(u)$ and $\Psi(u)$ are non-decreasing for $u \geq 0$, and $p$ is continuous. Moreover, assume that $\lim_{x \to \infty} \Psi(x) = \infty$ and $p(0)=0$. Let $f(x)$ be continuous on $[0, 1]$ and suppose that
\begin{align}
\notag \int_0^1 \int_0^1 \Psi\left( \frac{f(x) - f(y) }{p(x - y)} \right)dx dy \leq B < \infty.
\end{align}
Then, for all $s, t \in [0, 1]$, 
\begin{align}
\notag |f(s) - f(t) | \leq 8 \int_0^{|s - t|} \Psi^{-1} \left( \frac{4B}{u^2} \right) dp(u).
\end{align}
\end{proposition}
As an immediate consequence we obtain \textcolor{black}{a} Sobolev embedding theorem in the one-dimensional case. 
\begin{corollary}
\label{cor:sobolev-embedding}
\textcolor{black}{Let $q \geq 1$, and $\alpha q> 1$. Then there exists a constant $C_{\alpha,q} > 0$ such that}
\begin{align*}
|f(t)-f(s)| \leq C_{\alpha,\textcolor{black}{q}}|t-s|^{\alpha -\textcolor{black}{ 1/q}} \left(\int_0^1\int_0^1 \frac{|f(x)-f(y)|^{\textcolor{black}{q}}}{|x-y|^{\alpha \textcolor{black}{q} +1}}dxdy\right)^{\frac{1}{\textcolor{black}{q}}}
\end{align*}
\textcolor{black}{holds} for any continuous function $f$ on $[0,1]$ \textcolor{black}{and $s,t \in [0,1]$}.

\end{corollary}

The following multiparameter version of Proposition \ref{prop:GRR} was proved in \cite{hu-le2013}.
\begin{proposition}[Theorem 2.3 \cite{hu-le2013}]
\label{prop:GRR-field}
Let $\Psi(u)$ be a non-negative even function on $(- \infty, \infty)$ and $p_j(u),j=1,\ldots,n$ be non-negative even functions on $[-1,1]$. Assume that all $p_j(u),j=1,\ldots,n$ and $\Psi(u)$ are non-decreasing for $u \geq 0$, and that the functions $p_j$ are continuous. Moreover, assume that $\lim_{x \to \infty} \Psi(x) = \infty$ and $p(0)=0$. Let $f$ be a continuous function on $[0,1]^n$ and suppose that 
$$
B:=\int_{[0,1]^n}\int_{[0,1]^n} \Psi\left(\frac{|\Box_y^n f(x)|}{\prod_{j=1}^n \rho_j(|x_k-y_k|)}\right)dxdy < \infty.
$$
Then
$$
|\Box_y^n f(x)| \leq 8^n \int_{0}^{|s_1-t_1|}\ldots \int_0^{|s_n-t_n|} \Psi^{-1}\left(\frac{4^nB}{u_1^2\ldots u_n^2}\right)d\rho_1(u_1)\ldots d\rho_n(u_n).
$$
\end{proposition}
\begin{proof}[Proof of Theorem \ref{thm:sufficient-process}]
The bound \eqref{eq:continuity-general-process} follows directly from Proposition \ref{prop:GRR} with the choice $\Psi(x) = \exp\left(\beta |x|^{\frac{1}{\iota}}\right)$, which has the inverse $\Psi^{-1}(x) = \beta^{-\iota}\left(\log x\right)^{-\iota}$, and hence it suffices to prove the claim on the moments. Using Minkowski's integral inequality, we have, for any $p \geq 1,$
$$
\E B^p \leq \left[\int_0^1\int_0^1 \left(\E \exp\left(\beta p\left(\frac{|X_t-X_s|}{\rho(|t-s|)}\right)^{\frac{1}{\iota}}\right)\right)^{\frac{1}{p}}dsdt\right]^p,
$$
where, by \eqref{eq:hyper-process}, we have
\begin{align*}
\E\exp\left(\beta p \left(\frac{|X_t-X_s|}{\rho(|t-s|)}\right)^{\frac{1}{\iota}}\right) &= \E \sum_{k=0}^\infty \frac{(\beta p)^k}{k!}\left(\frac{|X_t-X_s|}{\rho(|t-s|)}\right)^{\frac{k}{\iota}} 
\leq \sum_{k=0}^\infty \frac{(\beta p)^k}{k!}C_0^{\frac{k}{\iota}}\left(\frac{k}{\iota}\right)^{k}\\
 =&\sum_{k=0}^\infty\left(\frac{\beta pC_0^{\frac{1}{\iota}}}{e\iota}\right)^k e^k \frac{k^k}{k!}. 
\end{align*}
From Stirling's approximation we get $k! = \Gamma(k+1) \sim \sqrt{2\pi k} k^ke^{-k}$ for large $k$, allowing us  to deduce that
\begin{align} \label{eq:exp-moments-finite}
\sum_{k=0}^\infty\left(\frac{\beta pC_0^{\frac{1}{\iota}}}{e\iota}\right)^k e^k \frac{k^k}{k!} 
\leq c \sum_{k=0}^\infty \left(\frac{\beta pC_0^{\frac{1}{\iota}}}{e\iota}\right)^k 
= \frac{c}{1-\frac{\beta pC_0^{\frac{1}{\iota}}}{e\iota}} < \infty,
\end{align}
since by assumption
$
\frac{\beta pC_0^{\frac{1}{\iota}}}{e\iota} < 1.
$
This completes the proof.
\end{proof}
\begin{proof}[Proof of Corollary \ref{cor:holder-process}]
The continuity of $X$ follows from Assumption \ref{assu:hyper-basic} together with $\E(X_t-X_s)^2 \leq |t-s|^{2\alpha}$ and the classical Kolmogorov continuity criterion. Now from Theorem \ref{thm:sufficient-process} we get
\begin{align*}
|X_t - X_s | &\leq 8\alpha \int_0^{|s - t|} \beta^{-\iota}\left(\log\left(\frac{4B}{u^2}\right)\right)^{\iota} u^{\alpha-1}du \\
&= 8\alpha\beta^{-\iota}|t-s|^\alpha\int_0^{1} \left(\log\left(\frac{4B}{|t-s|^2v^2}\right)\right)^{\iota}v^{\alpha-1}dv\\
&= 8\alpha\beta^{-\iota}|t-s|^\alpha\int_0^{1} \left(\log (4B)+2\log v^{-1} + 2\log |t-s|^{-1}\right)^{\iota}v^{\alpha-1}dv \\
&\leq 8\alpha\cdot 3^{\max(\iota-1,0)}\beta^{-\iota}|t-s|^\alpha \int_0^1 \big[\left(\log\max(4B,1)\right)^{\iota} + \left(2\log v^{-1}\right)^{\iota} \\ &\quad + \left(2\log |t-s|^{-1}\right)^{\iota} \big]v^{\alpha-1}dv,
\end{align*}
where we have used Jensen's inequality to obtain the last inequality. Hence we may set 
$$
C(\omega) = 8\cdot 3^{\max(\iota-1,0)}\beta^{-\iota}\left(\log \max(4B(\omega),1)\right)^{\iota} 
$$
and 
$$
C_d = 8\alpha\cdot 2^\iota\cdot 3^{\max(\iota-1,0)}\beta^{-\iota}\left[\alpha^{-1}+\int_0^1 \left(\log v^{-1}\right)^\iota v^{\alpha-1}dv\right]
$$
to obtain \eqref{eq:holder-process}.
For the claim $\E\exp(\beta_0 C(\omega)^{1/\iota})<\infty$, we retrace the arguments at the end of the proof of Theorem \ref{thm:sufficient-process}, and we obtain on the set $\{B>1/4\}$ that 
\begin{equation}\label{eq:exponential-moments-B}
    \exp(\beta_0 C(\omega)^{1/\iota}) 
    = 4 B(\omega)^{\beta_0\beta^{-1}\left(8\cdot 3^{\max(\iota-1,0)}\right)^{\frac{1}{\iota}}}
\end{equation}
which has finite moments provided that (see \eqref{eq:exp-moments-finite})
$$
\beta_0\beta^{-1}\left(8\cdot 3^{\max(\iota-1,0)}\right)^{\frac{1}{\iota}} \leq p < \frac{e\iota}{\beta C_0^{1/\iota}}.
$$
This translates into \eqref{eq:beta_0-process}.
Finally, \eqref{eq:lil} follows directly from \eqref{eq:holder-process}\textcolor{black}{,} completing the proof.
\end{proof}
\begin{proof}[Proof of Corollary \ref{cor:sufficient-process}]
By Corollary \ref{cor:holder-process} we have, for every $s,t\in I$, that
$$
|X_t-X_s| \leq C(\omega)|I|^\alpha + C_d\sup_{0\leq x\leq |I|}x^\alpha \left(\log 1/x\right)^\iota \leq C(\omega)|I|^\alpha + C_d e^{-\iota}\left(\frac{\iota}{\alpha}\right)^\iota,
$$
where we have used the fact that the function 
\begin{align}\label{eq:cor-f}
f(x) = x^\alpha (-\log x)^\iota    
\end{align}
attains its maximum in $[0,1]$ at $x^* = e^{-\iota}\left(\frac{\iota}{\alpha}  \right)^\iota$. Applying also Chebyshev's inequality, it follows that
    \begin{align*}
        P \Bigg(\sup_{t \in I} \abs{X_t-X_s} \geq u|I|^\alpha + C_d e^{-\iota}\left(\frac{\iota}{\alpha}\right)^\iota \Bigg) 
        &\leq 
         P\left( C(\omega) \geq u         \right) \\
        &\leq  
        e^{-\beta_0 u^{1/\iota}  } \E{e^{\beta_0 C(\omega)^{1/\iota}}}   \\
        &= 
        C(\beta_0) e^{-\beta_0 u^{1/\iota}},
        \end{align*}
where from \eqref{eq:exponential-moments-B} we set $C(\beta_0) = C(\beta_0,\beta) = \E[4B^{\beta_0 \beta^{-1}\left(8 \cdot 3^{\max(\iota-1,0)}\right)^{1/\iota}  }  ]<\infty$ for any $\beta_0$ satisfying \eqref{eq:beta_0-process}. This completes the proof.

\end{proof}
Before the proof of Theorem \ref{thm:holder-process-iff}, we need two additional lemmas. The first one is the well-known Paley\textcolor{black}{--}Zygmund inequality.
\begin{lemma}[Paley\textcolor{black}{--}Zygmund]
Let \(X\) be a non-negative random variable with finite variance, and \(\theta \in [0,1]\). Then
\begin{align}\label{eq:p-z}
P(X > \theta \, \mathbb{E}X) \geq (1-\theta)^2 \frac{[EX]^2}{EX^2}.
\end{align}
\end{lemma}
\begin{lemma}
\label{lem:lemma1-vastine}
Let \((F_i)_{i \in I}\) be a tight collection of non-negative random variables satisfying $\E F_i^4 \leq C(\E F_i^2)^2$ for all $i\in I$, where $C$ is independent of $i$. Then 
$$
\sup_{i\in I} \E F_i^2 < \infty.
$$
\end{lemma}
\begin{proof}
We choose \(X = F_i^2, \, \theta = \frac{1}{2}\) in \eqref{eq:p-z} to obtain
\begin{align*}
P(F_i^2 > \mathbb{E}F_i^2/2) \geq \frac{1}{4} \frac{[\E F_i^2]^2}{\E F_i^4} \geq \frac{1}{4C}.
\end{align*}
By assumption the collection \((F_i)_i\) is tight, and thus for any \(\varepsilon>0\) there exists a constant \(K_\varepsilon>0\) such that for all \(i \in I\) we have \(P(\lvert F_i \rvert > K_\varepsilon) < \varepsilon\). Choosing \(\varepsilon = \frac{1}{4C}.\) leads to 
\[
P(\lvert F_i \rvert > K_\varepsilon) < \frac{1}{4C} \leq P(F_i^2 > \mathbb{E}F_i^2/2),
\]
and hence \(\mathbb{E}F_i^2 < 2 K^2_\varepsilon\). Here by the tightness of the collection \((F_i)_{i \in I}\), the constant \(K_\varepsilon\) is uniform in \(i\), and hence it follows that \(\sup_{i \in I} \mathbb{E}F_i^2 < \infty, \) completing the proof.
\end{proof}
\begin{proof}[Proof of Theorem \ref{thm:holder-process-iff}]
Assuming \eqref{eq:holder-iff} and since $\epsilon>0$ is arbitrary, the H\"older continuity of any order $\gamma<\alpha$ follows directly from the Kolmogorov continuity criterion and Assumption \ref{assu:hyper-basic}. For the other direction, set 
$$
F_{s,t} = \frac{|X_t-X_s|}{|t-s|^{\alpha-\epsilon}}.
$$ 
As $F_{s,t}$ is a tight collection by H\"older continuity and satisfies $\E F^4_{s,t} \leq c \left[\E F^2_{s,t}\right]^2$, it follows from Lemma \ref{lem:lemma1-vastine} that then 
$$
\sup_{s,t} \E F_{s,t}^2 < \infty.
$$ 
That is, $\E(X_t-X_s)^2 \leq C_\epsilon |t-s|^{2\textcolor{black}{\alpha}-2\epsilon}$. Hence it remains to prove the existence of moments. By H\"older continuity, we have 
$$
|X_t-X_s| \leq C_\epsilon(\omega)|t-s|^{\alpha - \epsilon}.
$$
Note that the constant $C_\eps(\omega)$ depends also on the H\"older index $\alpha$ but this dependence is omitted in the notation for simplicity. On the other hand, using Corollary \ref{cor:sobolev-embedding} we obtain
$$
|X_t - X_s| \leq C_{\textcolor{black}{\gamma},p}|t-s|^{\gamma - 1/p}\left(\int_0^1 \int_0^1 \frac{|X_u-X_v|^p}{|u-v|^{1+\gamma p}}dudv\right)^{\frac{1}{p}}
$$
for any $\gamma$ and $p$ such that $\gamma p > 1$. By choosing $\gamma = \alpha - \frac{\epsilon}{2}$ and $p = \frac{2}{\epsilon}$ allows us to choose
$$
C_\epsilon(\omega) = C_\epsilon\left(\int_0^1 \int_0^1 \frac{|X_u-X_v|^{2/\epsilon}}{|u-v|^{2\alpha / \epsilon}}dudv\right)^{\frac{\epsilon}{2}}. 
$$
Moreover, \textcolor{black}{ as shown above, we have} \eqref{eq:holder-iff}. \textcolor{black}{That is,} we have, for any $\delta \textcolor{black}{ > 0}$ that 
$$
\E (X_t -X_s)^2 \leq C_\delta|t-s|^{2\textcolor{black}{\alpha}-\delta}.
$$
Arguing as in \cite{lauri2014} leads to, for every $q\geq \frac{\epsilon}{2}$ and $\delta < \frac{\epsilon}{2}$, 
$$
\E C^q_{\epsilon}(\omega) \leq 2C_0^q q^{q\iota}C_\delta \left(\frac{\epsilon}{2\delta}\right)^{\frac{q\epsilon}{2}}\left(1-\frac{\epsilon}{2\delta}\right)^{\frac{q\epsilon}{2}}.
$$
By expanding the exponential as in the proof of Theorem \ref{thm:sufficient-process} and retracing the argument of said proof, we finally obtain \eqref{eq:exp-moments} for sufficiently small $\beta$ (depending on the chosen $\delta$ and $\epsilon$). This completes the whole proof.
\end{proof}

We are now ready to present proofs for our results in the case of fields. As they follow essentially from the same arguments, we only sketch some essential arguments that are required.
\begin{proof}[Proof of Proposition \ref{prop:field-holder}]
\textcolor{black}{We begin by noting that the hypotheses of the Proposition yield that either $X$ is immediately Hölder continuous, or that \eqref{eq:holder-iff-multi-ass2} with Assumption \ref{assu:hyper-field} together imply (via Kolmogorov--Chentshov) that $X$ is a continuous field on $[0,1]^n$. Hence we have $t \mapsto X_{t} \in L^p([0,1]^n)$ almost surely for any $p \geq 1$, where $L^p([0,1]^n)$ denotes the usual (deterministic) $L^p$-space. Consequently, we may combine \cite[p. 563 Eq. (8.4)]{hitchhiker} and \cite[p. 564 Eq. (8.8)]{hitchhiker} \footnote{Actually, the stated upper bound of Eq. (8.4) in \cite{hitchhiker} includes the $L^p([0,1]^n)$ norm of $X$. However, by examining the derivation, see \cite[p. 563 Eq. (8.3)]{hitchhiker}, one observes that one can only work without $L^p([0,1]^n)$ norm which then leads to our bound.} to get  
$$
|X_t-X_s| \leq C|t-s|^{\frac{\gamma p - n}{p}} \left(\int_{[0,1]^n}\int_{[0,1]^n} \frac{|X_u-X_v|^p}{|u-v|^{n+\gamma p}}dudv\right)^{\frac{1}{p}}
$$
for $\gamma p > n$. The claim follows from this by using the same arguments as in the proof of Theorem \ref{thm:holder-process-iff}.}

\end{proof}
\begin{proof}[Proof of Theorem \ref{thm:sufficient-field}]
The proof follows analogously to the proof of Theorem \ref{thm:sufficient-process}. \textcolor{black}{The stated upper bound follows from} the multiparameter Garsia\textcolor{black}{--}Rodemich\textcolor{black}{--}Rumsey inequality, Proposition \ref{prop:GRR-field}, \textcolor{black}{with the choice $\Psi(x)=\exp\left(   \beta \abs{x}^{\frac{1}{\iota}}\right)$. The claimed moments are obtained as in the proof of Theorem \ref{thm:sufficient-process}, except that the increment $X_t-X_s$ is replaced with the rectangular increment $\Box_t^n X(s)$ in the natural way, and similarly for $\prod_{j=1}^n \rho_j$ replacing $\rho$. } 
\end{proof}
\begin{proof}[Proof of Corollary \ref{cor:holder-field}]
Following the proof of Corollary \ref{cor:holder-process} we obtain that we may choose 
$$
C(\omega) = 8^n\cdot 3^{\max(\iota-1,0)}\beta^{-\iota}\left(\log \max(4^n B(\omega),1)\right)^{\iota} 
$$
and 
$$
C_{d} = 8^n\prod_{j=1}^n\alpha_j\cdot 2^\iota\cdot 3^{\max(\iota-1,0)}\beta^{-\iota}\left[\prod_{j=1}^n\alpha_j^{-1}+\int_{[0,1]^n} \left(\sum_{j=1}^n\log v_j^{-1}\right)^\iota \prod_{j=1}^nv_j^{\alpha_j-1}dv_1\ldots dv_n\right].
$$
The rest of the proof goes analogously to the proof of Corollary \ref{cor:holder-process}.
\end{proof}
The proof of Corollary \ref{cor:sufficient-field} 
\textcolor{black}{runs analogously to its one-parameter counterpart, Corollary \ref{cor:sufficient-process}, with the modification that in said proof we were able to determine explicitly the maximum of $f$ in \eqref{eq:cor-f}. This is now replaced with $\widetilde{C}$ as in \eqref{eq:c-tilde} in the statement of Corollary \ref{cor:sufficient-field}, and we note that it is clear that $\widetilde{C} < \infty$.   }

For Theorem \ref{thm:holder-field-iff}, \textcolor{black}{we comment that in Proposition \ref{prop:GRR-field} choosing $\Psi(x)=x^q, \rho(u)=u^{\gamma+1/q}$, with $\gamma = (\gamma_1, \hdots, \gamma_n)$ with $\gamma_j q > 1$  we obtain the following natural extension of Corollary \ref{cor:sobolev-embedding} to the multiparametric case:}

\begin{align}
    \textcolor{black}{\abs{\Box_t^n f(s)} \leq C_{\alpha,q} \prod_{j=1}^n \abs{t_j-s_j}^{\alpha_j-1/q} \left( \int_{[0,1]^n} \int_{[0,1]^n} \frac{\abs{\Box_y^n f(x)}^{q}}{\prod_{j=1}^n \abs{x_j-y_j}^{\alpha_j q+n}} \, dx dy \right)^{\frac{1}{q}}.}
\end{align}
\textcolor{black}{With this inequality at our disposal, the remainder of the proof is analogous to that of its one-dimensional variant in Theorem \ref{thm:holder-process-iff}; we simply replace $\rho$ and $X_t-X_s$ with $\prod_{j=1}^n \rho_j$ and $\Box_t^nX(s)$, respectively.}

\bibliographystyle{plain}

\bibliography{main-revised-no-red}

\end{document}